\pdfpkmode={lexmarkr}
\pdfmapfile{}

\documentclass[a4paper,11pt,reqno]{amsart}
\usepackage[english]{babel}
\usepackage[latin1]{inputenc}
\usepackage{graphicx}
\usepackage{url}
\usepackage{mathtools,amssymb}
\usepackage[shortlabels]{enumitem}
\usepackage{csquotes}
\usepackage[margin=1.40in]{geometry}

\newtheorem{theorem}{Theorem}[section]
\newtheorem{proposition}[theorem]{Proposition}
\newtheorem{lemma}[theorem]{Lemma}

\theoremstyle{definition}
\newtheorem{remark}[theorem]{Remark}

\newcommand \Ab {{\bf A}}

\numberwithin{equation}{section}
\numberwithin{table}{section}
\numberwithin{figure}{section}

\DeclareMathOperator{\curl}{curl}
\DeclareMathOperator{\Div}{div}
\DeclareMathOperator{\Hess}{Hess}
\DeclareMathOperator{\Area}{Area}

\title[On the ground state energy of the Pauli operator]{On the 
semi-classical analysis\\ of the groundstate energy\\ of the Dirichlet Pauli operator}

\author{Bernard Helffer}
\author{Mikael Persson Sundqvist}
\address[Bernard Helffer]{ Laboratoire de Math\'ematiques Jean Leray, 
Universit\'e de Nantes,  2 rue de la Houssini\`ere, 44322 Nantes, France
and Laboratoire de
Math\'{e}matiques, Universit\'{e} Paris-Sud, France. 
}
\email{bernard.helffer@univ-nantes.fr}
\address[Mikael Persson Sundqvist]{Lund University, Department of Mathematical 
Sciences, Box 118, 221\ 00 Lund, Sweden.}
\email{mickep@maths.lth.se}

\begin{document}

\begin{abstract}
 We discuss the results of a recent paper by Ekholm, Kova\v{r}\'ik and Portmann in 
connection with a question of C. Guillarmou about the semiclassical expansion
of the lowest eigenvalue of the Pauli operator with Dirichlet conditions. 
We exhibit connections with the properties of the torsion function in mechanics, 
the exit time of a Brownian motion and the analysis of the low eigenvalues of 
some Witten Laplacian.
\end{abstract}
\maketitle
\section{Introduction}
We study the Dirichlet realization of the Pauli operator
$P=(P_+,P_-)$,
\[
P_\pm:= (hD_{x_1} -A_1)^2 + (hD_{x_2}-A_2)^2 \pm h B(x)\,,
\]
in a bounded, regular and open set $\Omega \subset \mathbb R^2$. 
Here $D_{x_j} = -i \partial _{x_j}$ for $j=1$, $2$. We assume that the magnetic 
potential $(A_1,A_2)$ belongs to $C^\infty(\overline{\Omega})^2$ and denote by
$B(x) = \partial_{x_1} A_2 - \partial_{x_2} A_1$ the associated magnetic field. 

To be more precise, we are interested in the smallest eigenvalue of $P$ as 
$h\to 0^+$. Since the Pauli operator is the square of a Dirac operator this
eigenvalue is non-negative. By the diamagnetic
inequality it follows that if $B(x)\leq 0$ then the smallest 
eigenvalue $\lambda_{P_-}^D(h,B,\Omega)$ of $P_-$ satisfies
\begin{equation}
\label{eq:diamagnetic}
\lambda_{P_-}^D(h,B,\Omega)\geq h^2\lambda^D(\Omega).
\end{equation}
Here $\lambda^D(\Omega)$ denotes the smallest eigenvalue of the 
Dirichlet Laplacian $-\Delta$ in $\Omega$. By symmetry, the same statement
holds for $\lambda_{P_+}^D(h,B,\Omega)$, the smallest eigenvalue of $P_+$, in
case $B(x)\geq 0$. We will focus our study to the component $P_-$ under the assumption
\begin{equation}\label{Ass1}
\{ x \in \Omega : B(x) >0\} \neq \varnothing\,,
\end{equation}
sometimes with the stronger condition that $B(x)> 0$. This means that we
cannot directly apply the diamagnetic inequality 
to conclude the lower bound~\eqref{eq:diamagnetic}. In fact, we will see
that $\lambda_{P_-}^D(h,B,\Omega)$ is exponentially small.

In~\cite{EKP}, T.~Ekholm, H.~Kova\v{r}\'ik and F.~Portmann give a universal lower 
bound which can be formulated in the semiclassical context in the following way:

\begin{theorem}[{\cite[Theorem~2.1]{EKP}}]\label{thmEKP}
If $B$ does not vanish identically in the simply connected domain $\Omega$ there 
exists $\epsilon(B,\Omega) >0$ such that, $\forall h >0$,
\begin{equation}
 \lambda_{P_-}^D(h,B,\Omega) \geq   \lambda^D(\Omega) \, h^2\,  
\exp \bigl(- \epsilon (B,\Omega) /h\bigr) \,.
 \end{equation}
 \end{theorem}
The proof of the theorem gives a way of computing some upper bound for 
$\epsilon (B,\Omega)$, by considering the oscillation of scalar potentials
$\psi$, i.e. solutions of $\Delta \psi =B$, and optimizing over $\psi$.  
Although the authors treat interesting examples, they do not give a systematic 
approach for determining the optimal lower bound.

Scalar potentials will also play a main role for us, and we will
fix a canonical one associated with a specific choice of the gauge for the 
magnetic potential, as used for example in superconductivity theory.
Thus, below, $\psi$ will denote the solution of the Poisson problem with
Dirichlet boundary condition,
\begin{equation}
\label{eq:psi}
\begin{cases}
\Delta\psi=B & \text{in $\Omega$}\\
\psi=0 & \text{on $\partial\Omega$}.
\end{cases}
\end{equation}

It turns out that this $\psi$ gives the minimal oscillation discussed
in~\cite{EKP} under the further condition that $B(x)>0$. In order to answer 
negatively a question of C.~Guillarmou, who was asking for an example for 
which $B(x)>0$ and there exists $N_0$ and $c_N>0$ such that 
$\lambda_{P_-}^D(h,B,\Omega)\geq c_N h^N$, we also discuss upper bounds for the
ground state energy, and state our main result:
\begin{theorem}
\label{th1.3}
Assume that $B(x) >0$ in $\Omega$ and that $\psi$ satisfies~\eqref{eq:psi}. Then 
\begin{equation}
\label{aa2}
\lim_{h\to 0^+} h \log  \lambda_{P_-}^D(h,B,\Omega) = 2 \inf_\Omega \psi\,.
\end{equation}
\end{theorem}

In particular, this shows that when $B(x) >0$, the optimal $\epsilon (B,\Omega)$ 
is given by the analysis of~\eqref{eq:psi}, justifying a posteriori our choice 
of $\psi$. Moreover, it raises the question
on finding $\inf_\Omega\psi$. Kawohl proves in~\cite{K1,K2} that
if $\Omega$ is strictly convex and $B(x)$ constant then the infimum is attained 
at a unique point $x_0\in\Omega$. We discuss this, and further properties
of $\psi$ in Section~\ref{sec:torsion}.

Our lower bound leading to Theorem~\ref{th1.3} is true for all $h>0$. 
It is stated and proved in Section~\ref{sec:EKP}, see Theorem~\ref{th2.1}.

Under additional assumptions on $\Omega$ or $B(x)$, we will propose a deeper 
analysis giving more accurate lower bounds or upper bounds and giving in 
particular the equivalent of $\lambda_{P_-}^D(h,B,\Omega)$ when $\Omega$ is the 
disk and $B(x)$ is constant. In fact, we are able to give the main asymptotic
term for all exponentially small eigenvalues. This also permits to clarify some 
miss-statements \cite{HeMo1, FH} appearing in the literature and to improve 
(or correct) the results of~\cite{E1,E2,EV},~\cite{HeMo1},~\cite{FH} and~\cite{EKP}.

\begin{theorem}
\label{thm:Pauli-disk}
Assume that $\Omega=D(0,R)$ is the disk of radius $R$, and that the magnetic
field $B(x)$ is constant, $B(x)=B>0$.
Let $\lambda^D_{P_-,0}(h)$ denote the smallest eigenvalues of $P_-$. 
Then, as $h\to 0^+$,
\begin{equation}\label{asympdisk}
\lambda^D_{P_-,0}(h)=B^2R^2 e^{-BR^2/2h}(1+\mathcal O(h)).
\end{equation}
\end{theorem}

The proof of this theorem, given in Section~\ref{sec:disk}, relies heavily on 
the rotational symmetry of the ground state which was proven by L. Erd\H{o}s in~\cite{E1}.
In fact, we make an orthogonal decomposition
of $P_-$ using angular momentum. The upper bound comes from a very accurate 
quasimode. For each angular momentum we obtain a spectral gap, enabling us to
use the Temple inequality with the quasimode to get also the matching lower 
bound. It is the lack of this spectral gap that prohibits us to find the
corresponding lower bound for general domains $\Omega$.

In Section~\ref{ss6.1}, we recall a statement of L.~Erd\H{o}s permitting  in 
the constant magnetic field case a direct comparison between 
$\lambda^D_{P_-} (h,B,\Omega)$ and  $\lambda^D_{P_-} (h,B,D(0,R))$
where $D(0,R)$ is the disk of same area as $\Omega$.

In Section~\ref{sec:torsion} we collect some general properties of
the scalar potential $\psi$ in connection with the torsion problem.

We close this introduction with several remarks.

\begin{remark}
The case $B(x)\leq0\,$, mentioned above, can be analyzed further under 
additional conditions (for example, one could think to start with the case 
when $B(x)$ has a non degenerate negative  minimum at a point of $\Omega$). 
We refer to Helffer--Mohamed~\cite{HeMo0}, Helffer--Morame~\cite{HeMo1}, 
Helffer--Kordyukov~\cite{HK} and Raymond--Vu Ngoc~\cite{RVN} for the analysis of 
the Schr\"odinger with magnetic field, and note that the addition of the term 
$-h B(x)$ in $P_-$ can be controlled in their analysis. 
\end{remark}

\begin{remark}
When $B(x)=2\,$, the solution $\psi$ of~\eqref{eq:psi} appears to be $-f_\tau$, where 
$f_\tau$ is the so-called torsion function which plays an important role in Mechanics. 
For this reason there are a lot of treated examples in the engineering 
literature and a lot of mathematical studies, starting from the fifties with 
P\'olya--Szeg\"o \cite{PS}. This permits in particular to improve the applications 
given in \cite{EKP}.
\end{remark}

\begin{remark}
The problem we study is quite close  with the question of analyzing the 
smallest eigenvalue of the Dirichlet realization of:
\[
C_0^\infty(\Omega)\ni v 
\mapsto h^2 \,\int_{\Omega}\left|\nabla  v (x)\right|^2\;e^{-2f(x)/h}~dx\;.
\]
For this case, we can mention  Theorem~7.4 in \cite{FrWe}, 
which says (in particular) that, if $f$ has a unique non-degenerate local 
minimum $x_{min}$, then the lowest eigenvalue $\lambda_1(h)$
of the Dirichlet realization $\Delta^{(0)}_{f,h}$ in $\Omega$ satisfies:
\begin{equation}
\lim_{h\rightarrow 0} - h \,\log \lambda_1(h)= \inf_{x\in \partial \Omega} (f(x) - f(x_{min}))\;.
\end{equation}
More precise or general results (prefactors) are given in~\cite{BEGK, BGK,HeNi}. 
This is connected with the semi-classical analysis of Witten 
Laplacians~\cite{W, HS1,HS2, CFKS, Sim, HKN}.
\end{remark}

\begin{remark}
It would be interesting to analyze flux effects in the case of a non simply 
connected domain. We hope to come back to this question in another work.
\end{remark}

\subsection*{Acknowledgements}
The first author was aware of the question analyzed in~\cite{EKP} at a meeting 
in Oberwolfach organized in 2014 by V.~Bonnaillie-No\"el, H.~Kova\v{r}\'ik  
and K.~Pankrashkin where the authors present their work. He would like to also 
thank D. Bucur,  L.~Erd\H{o}s,  C.~Guillarmou, and B. Kawohl  for useful discussions leading 
to this paper.

\section{Preliminaries}

Following what is done for 
example in superconductivity (see~\cite{FH}), we can, possibly after a gauge 
transformation, assume that the magnetic vector potential $\Ab =(A_1,A_2)$ 
satisfies
\begin{equation}\label{eq:5a}
\begin{cases}
\curl \Ab=B \text{ and }\Div\Ab=0& \text{in $\Omega$}\\
\Ab\cdot\nu = 0 & \text{on $\partial\Omega$}\,.
\end{cases}
\end{equation}
If this is not satisfied for say $\Ab_0$ satisfying $\curl \Ab_0=B$, we can 
construct 
\[
\Ab = \Ab_0 + \nabla \phi
\]
satisfying in addition \eqref{eq:5a}, by choosing $\phi$ as a solution of
\[
\begin{cases}
-\Delta \phi = \Div \Ab_0 & \text{ in } \Omega\\
\nabla \phi\cdot \nu = - \Ab_0\cdot \nu & \mbox{ on } \partial \Omega\,,
\end{cases}
\]
which is unique if we add the condition $\int_\Omega \phi (x)\, dx =0\,$.
In this case there exists a scalar potential $\psi$ such that
\begin{equation}\label{eq:5}
A_1 = - \partial_{x_2} \psi\,,\, A_2 = \partial_{x_1} \psi\,.
\end{equation}
The condition $ \Ab \cdot \nu =0 \mbox{ on } \partial \Omega$ 
implies that $\psi =\psi_i$ is constant on each connected component 
$\Gamma_i$ of the boundary.
Hence, if  in addition  $\Omega$ is simply connected,  
there exists a unique $\psi$ such that \eqref{eq:5} is satisfied 
and~\eqref{eq:psi} holds.
In this case, we should write instead 
\begin{equation}\label{nsc}
\Delta \psi = B\text{ in } \Omega\,,\,  \psi =\psi_i - \sup \psi_i   \text{ on } \Gamma_i\,.
\end{equation}
The link with the circulations of $\Ab$ along each $\Gamma_i$ is analyzed at least formally in~\cite{JA}.

If $B(x)>0$, the minimax principle shows that $\psi \leq 0\,$, and, if 
$\Omega$ is simply connected,  because $B$ is not identically $0$, there exists 
at least a point $x_0 \in \Omega$ such that 
\begin{equation}
\psi (x_0) = \inf \psi < 0\,.
\end{equation}
We write
\[
\psi_{min}:=\inf \psi \,.
\]
In the non simply connected case, the infimum can be attained at one 
component of the boundary.

\section{The lower bound of Ekholm--Kova\v{r}\'ik--Portmann revisited.}
\label{sec:EKP}

We come back to the scheme of proof of the lower bound in~\cite{EKP}, and
state a more explicit bound for positive magnetic fields:
\begin{theorem}\label{th2.1}
Assume that $B(x)>0$ in $\Omega\,$. Then
\begin{equation}
 \lambda_{P_-}^D(h,B,\Omega) \geq h^2  \lambda^D(\Omega) 
\, \exp \Bigl(\frac{2 \psi_{min}}{h}\Bigr) \,.
\end{equation}
\end{theorem}

\begin{proof}
We bound the quadratic form from below. Let
\[
\Psi = \psi -\psi_{min}\,.
\]
Then
\[
0 \leq \Psi \leq -\psi_{min}\,.
\]
We write
\begin{equation}
\label{eq:uv}
u = \exp \Bigl(- \frac{\Psi }{h}\Bigr)\,  v\,,
\end{equation}
and use following identity (see  \cite{EV} and (2.4) in \cite{EKP})
\begin{equation}\label{eq:pr1}
%\| (hD-\Ab) u\|^2 -h \int_\Omega B(x) |u(x)|^2\, dx  
\langle u,P_-u\rangle
= h^2 \int_\Omega \exp \Bigl(- 2 \frac{\Psi}{h}\Bigr) |(\partial_{x_1} + i \partial_{x_2}) v|^2\, dx\,,
\end{equation}
valid if $\Ab$ satisfies \eqref{eq:5a}. With $u$ (and consequently  $v$) in 
$H_0^1(\Omega)$, we get
\begin{equation}\label{eq:pr2}
\begin{aligned}
%\| (hD-\Ab) u\|^2 -h \int_\Omega B(x) |u(x)|^2\, dx 
\langle u,P_-u\rangle
& \geq h^2 \exp \Bigl(\frac{2 \psi_{min}}{h}\Bigr)  \int_\Omega |(\partial_{x_1} + i \partial_{x_2}) v|^2\, dx \\
& \geq h^2 \exp \Bigl(\frac{2 \psi_{min}}{h}\Bigr)  \int_\Omega |\nabla  v (x)|^2\, dx\\
& \geq h^2 \exp \Bigl(\frac{2 \psi_{min}}{h}\Bigr)  \lambda^D(\Omega) \int_\Omega |v(x)|^2 dx\\
& \geq h^2 \exp \Bigl(\frac{2 \psi_{min}}{h}\Bigr)  \lambda^D(\Omega) \int_\Omega |u(x)|^2 dx\,.
\end{aligned}
\end{equation}
Here we have used the Dirichlet condition on $v$ to justify the integration by 
part in the second step.
\end{proof}
Note that the statement is much more explicit than in~\cite{EKP}. It is unclear 
at this stage that it is accurate (we just make a choice of one $\Psi$ and 
in~\cite{EKP} there was a minimization over all $\Psi$ satisfying 
$\Delta \Psi =B$). But note also that our sign condition on the magnetic field
is much stronger.

\section{Upper bounds in the simply connected case}
\label{sec:upper}

With the explicit choice of $\psi$ from~\eqref{eq:psi}, it is straightforward 
to get:
 \begin{proposition}
Assume that $\Omega$ is simply connected and $B(x)>0$ in $\Omega$. Then, for 
any $\eta >0\,$, there exists a $C_\eta >0$ such that
\begin{equation}
 \lambda_{P_-}^D(h,B,\Omega) \leq C_\eta  \exp \Bigl(\frac{2 \psi_{min}}{h}\Bigr)
\exp \Bigl(\frac{2\eta}{ h}\Bigr)\,.
\end{equation}
\end{proposition}
The proof is obtained by taking as  trial state 
$u= \exp\bigl( - \frac{\psi}{h}\bigr)\,  v_\eta$, 
with $v_\eta $ compactly supported in $\Omega$ and $v_\eta$ being equal to $1$ 
outside a sufficiently small neighborhood of the boundary and implementing this 
trial state in~\eqref{eq:pr1}. One concludes by the max-min principle. This proof 
does not work in the non simply connected case.

The first idea for improvement is to take $\eta = h$ and to control with respect 
to $h$, but it appears better to  consider the trial state
\begin{equation}
\label{optqm}
u=\exp \Bigl(-\frac{\psi}{h}\Bigr)  -\exp \Bigl(\frac{\psi}{h}\Bigr)\,.
\end{equation}
When $\Omega$ is simply connected $u$ evidently satisfies 
the Dirichlet condition. Using the substitution~\eqref{eq:uv},
\[
v = 1 - \exp \Bigl(\frac{2\psi}{h}\Bigr)\,.
\]
Thus,
\begin{equation*}
\begin{aligned}
\frac{h^2\int_{\Omega}\exp\bigl(-\frac{2\psi}{h}\bigr)|(\partial_{x_1}+i\partial_{x_2})v|^2\,dx}{\int_{\Omega}\exp\bigl(-\frac{2\psi}{h}\bigr)|v|^2\,dx} & 
=  4  \frac{\int_{\Omega}\exp\bigl(\frac{2\psi}{h}\bigr)|(\partial_{x_1}+i\partial_{x_2})\psi|^2\,dx}{\int_{\Omega}\exp\bigl(-\frac{2\psi}{h}\bigr)|v|^2\,dx}  \\
& =  4  \frac{\int_{\Omega}\exp\bigl(\frac{2\psi}{h}\bigr)|\nabla \psi|^2\,dx}{\int_{\Omega}\exp\bigl(-\frac{2\psi}{h}\bigr) |v|^2\,dx}  \\
&=  2 h   \frac{\int_{\Omega}\bigl(\nabla \exp\bigl(\frac{2\psi}{h}\bigr)\bigr) \cdot \nabla \psi \,dx}{\int_{\Omega}\exp\bigl(-\frac{2\psi}{h}\bigr)|v|^2\,dx}  \\
\end{aligned}
\end{equation*}
We integrate by parts in the numerator,
\begin{equation}\label{upbnd1}
\begin{aligned}
\int_{\Omega}\Bigl(\nabla \exp\Bigl(\frac{2\psi}{h}\Bigr)\Bigr) \cdot \nabla \psi \,dx & = \int_{\partial \Omega} \nu \cdot \nabla \psi    -\int_{\Omega} \exp\Bigl(\frac{2\psi}{h}\Bigr) B(x) \,dx \\
& = \int_\Omega B(x) \, dx   -\int_{\Omega} \exp\Bigl(\frac{2\psi}{h}\Bigr) B(x)  \,dx\\
& \leq \int_\Omega B(x) \, dx \,.
\end{aligned}
\end{equation}
In the simply connected case, we know by Hopf Lemma that $\nu\cdot \nabla \psi$ 
does not vanish at the boundary. This implies, using the Laplace integral 
method in  a tubular neighborhood of the boundary, that 
\[
\int_{\Omega} \exp\Bigl(\frac{2\psi}{h}\Bigr) B(x) \,dx = \mathcal O (h)\,.
\]
This suggests that the upper bound in~\eqref{upbnd1} is rather sharp, but we 
will not use this property.

We turn to the denominator. In case there is a unique point $x_0\in\Omega$ where
$\psi$ attains its minimum, we have the following lower bound, for some 
$\epsilon>0$,
\begin{equation}
\label{minden}
\begin{aligned}
\int_{\Omega}\exp\Bigl(-\frac{2\psi}{h}\Bigr)|v|^2\,dx  
& \geq \int_{D(x_0,\epsilon)}\exp\Bigl(-\frac{2\psi}{h}\Bigr)|v|^2\,dx\\
& \sim \exp \Bigl(\frac{-2\psi_{min}}{h}\Bigr) \int_{D(x_0,\epsilon)}  
\exp\Bigl( -  \frac{2(\psi(x) -\psi_{min})}{h}\Bigr) \,dx\,.
\end{aligned}
\end{equation}
If there are a finite number~$N$ of points $x_j$ in $\Omega$ such that 
$\psi(x_j) =\psi_{min}$,
\begin{align}
\label{mindenN}
\int_{\Omega}\exp\Bigl(-\frac{2\psi}{h}\Bigr)|v|^2\,dx  
& \geq  \sum_{j=1}^{N} \int_{D(x_j,\epsilon)}\exp\Bigl(-\frac{2\psi}{h}\Bigr)|v|^2\,dx\\
&\notag \sim \sum_{j=1}^{N} \exp \Bigl(\frac{-2\psi_{min}}{h}\Bigr) 
\int_{D(x_j,\epsilon)}  \exp\Bigl( -  \frac{2(\psi(x) -\psi_{min})}{h}\Bigr) \,dx\,,
\end{align}
where $\epsilon>0$ is chosen such that the balls  $D(x_j,\epsilon)$ are disjoint.

If  $\Hess\psi(x_j)$ is positive, using the Laplace integral method, we get
\begin{equation}
\label{asympfin}
\int_{D(x_j,\epsilon)}  \exp\Bigl( -  \frac{2(\psi (x)-\psi_{min})}{h}\Bigr) \,dx 
\sim \pi  h (\det(\Hess \psi(x_j)))^{-\frac 12} \,.
\end{equation}

If $\Hess\psi(x_j)$ is not positive definite, then it necessarily has one zero 
eigenvalue and the other one equals $B(x_j)$. After a change of variable, we 
can assume that, with $x_j=(x_{j1},x_{j2})$, there exists $C_\epsilon$ such that
\[
\psi (x,y)-\psi_{min} \leq C_\epsilon   ( (x- x_{j1}) ^2 + (y-x_{j2})^4) \,, \, 
\text{ in }  D(x_j,\epsilon)\,.
\]
In this case we get, the existence of $d_\epsilon >0$ such that, as $h\to 0\,$, 
\begin{equation}
\label{mindena}
\int_{D(x_j,\epsilon)}  \exp\Bigl(-\frac{2(\psi (x)-\psi_{min})}{h}\Bigr) \,dx 
\geq  d_\epsilon  h^\frac 34.
\end{equation}

This leads us to consider two cases:
\begin{enumerate}[(1),leftmargin=*]
\item\label{case1} For all the $x_j$ such that $\psi (x_j)=\psi_{min}$, 
$\Hess \psi (x_j)$ positive definite.
\item\label{case2} There exists $x_0$ such that $\psi (x_0) =\psi_{min}$ and 
$\Hess \psi (x_0)$ is degenerate.
\end{enumerate}
Note that Case~\ref{case1} contains the case when $\Omega$ is convex and 
$B(x) = B >0$. In this case there is a unique $x_j$ 
(see Proposition~\ref{propkaw} below).

The lower bounds in~\eqref{minden} and~\eqref{mindenN} together with the upper 
bound deduced from~\eqref{upbnd1} lead to the following upper bounds.
\begin{theorem}\label{th3.3}
Assume that $\Omega$ is simply connected and that $B(x)>0$ in $\Omega$. Denote
by $\Phi=\frac{1}{2\pi}\int_{\Omega}B(x)\,dx$ the flux of the magnetic field
$B(x)$ through $\Omega$.
\begin{enumerate}[a),leftmargin=*]
\item In case~\ref{case1}, as $h\to 0^+$, 
\begin{equation}
 \lambda_{P_-}^D(h,B,\Omega) \leq  4\Phi \,  \biggl[
 \sum_{j=1}^{N} \left( \det \,(\Hess \psi(x_j))\right)^{-\frac 12}\, \biggr]^{-1} 
\exp\Bigl( \frac{2 \psi_{min}}{h}\Bigr) (1+o(1)) \,.
\end{equation}
\item In case~\ref{case2}, as $h\to 0^+$,
\begin{equation}
 \lambda_{P_-}^D(h,B,\Omega) = \mathcal O (h^{\frac 14} ) \exp\Bigl( \frac{2 \psi_{min}}{h}\Bigr)  \,.
\end{equation}
\end{enumerate}
\end{theorem} 
In the case of the disk $D(0,R)$, we get, as $h\to 0^+\,$,
\begin{equation}\label{upbddisk1}
 \lambda_{P_-}^D(h,B,D(0,R))\leq B^2 R^2 \exp (-B R^2/2h) (1 + o(1))\,.
\end{equation}
We will see in Section~\ref{sec:disk} that it is optimal for the disk. 
Unfortunately, the proof is specific of the disk.

\begin{remark}
\begin{itemize}[--,leftmargin=*]
\item[]
\item
When $\Omega$ is not simply connected, we can use the domain monotonicity of 
the Dirichlet problem and apply the previous result for any simply connected 
domain $\widetilde \Omega$ contained in $\Omega$.
The natural question is then to find the optimal domain and it is unclear if 
this would lead to the optimal decay.
\item
The same monotonicity argument could be used  for treating the case when $B$ 
is changing sign. We should add in this case for the choice of 
$\widetilde \Omega$  the condition that $B >0$ on $\widetilde \Omega$.
\end{itemize}
\end{remark}

\section{The case of  the disk in the constant magnetic case revisited.}
\label{sec:disk}
In this section we work with constant magnetic field $B(x)=B>0$ in the 
disk $\Omega=D(0,R)$, and our aim is to prove Theorem~\ref{thm:Pauli-disk}. 
We actually present a more general analysis of the spectrum.

We introduce polar coordinates $(r,\theta)$ via $x_1=r\cos\theta$ and
$x_2=r\sin\theta$. As is well-known, the variables separate,
and we are led to study the infinite family of operators 
$\mathcal P_m(h)$, $m\in\mathbb Z$, where
\[
\mathcal P_m(h)=h^2\Bigl[-\frac{d^2}{dr^2}-\frac{1}{r}\frac{d}{dr}
+\Bigl(\frac{m}{r}-\frac{Br}{2h}\Bigr)^2\Bigl]-hB.
\]
The spectrum of $P_-$ is the union of the spectrum of 
the operators $\mathcal P_m(h)$. 
\begin{theorem}
\label{thm:Pauli-m} For $m\in \mathbb Z$, let 
$\lambda_{m,0}^D(h)\leq \lambda_{m,1}^D(h)\leq\ldots$ denote 
the increasing sequence of eigenvalues of $\mathcal P_m(h)$. 
\begin{enumerate}[a),leftmargin=*]
\item\label{it:a} If $m<0$ then
$\lambda_{m,0}^D(h)\geq (2|m|-1)hB\,$.
\item\label{it:b} If $m\geq 0$ then 
$\lambda_{m,k}^D(h)\geq 2hBk\,$.

\noindent In particular, the second
eigenvalue satisfies $\lambda_{m,1}^D(h)\geq 2hB\,$.
\item\label{it:c} If $m\geq 0$, then, as $h\to 0^+\,$,
\[
\lambda_{m,0}^D(h)=2hB\frac{(BR^2/2h)^{m+1}}{m!}e^{-BR^2/2h}(1+O(h))\,.
\]
\end{enumerate}
\end{theorem}

\begin{proof}[Proof of Theorem~\ref{thm:Pauli-disk}]
We know from~\cite{E1} that the lowest eigenvalue comes from the angular
momentum $m=0$, i.e.
$\lambda_{P_-,0}^D(h)=\lambda_{0,0}^D(h)$. Thus, the result in 
Theorem~\ref{thm:Pauli-disk} is a direct consequence of statement~\ref{it:c} in 
Theorem~\ref{thm:Pauli-m}.
\end{proof}

\begin{proof}[Proof of Theorem~\ref{thm:Pauli-m}]
We divide the proof into several steps.

\subsubsection*{Proof of~\ref{it:a}}
This follows directly from an estimate of the potential,
\[
h^2\Bigl(\frac{m}{r}-\frac{Br}{2h}\Bigr)^2
=\Bigl(\frac{hm}{r}\Bigr)^2+\Bigl(\frac{Br}{2}\Bigl)^2-mhB\geq 2|m|hB\,,
\]
and a comparison of quadratic forms.

\subsubsection*{Proof of~\ref{it:b}}
We will use the fact that the eigenfunctions of $\mathcal P_m$ 
can be expressed in terms of Kummer functions together with a result 
on zeros of these functions, explained below.

The Kummer differential equation reads
\[
z\frac{d^2w}{dz^2}+(b-z)\frac{dw}{dz}-aw=0\,.
\]
One solution to Kummer's equation, that is regular at $z=0\,$, is given by
\[
\mathbf M(a,b,z)=\sum_{s=0}^{+\infty}\frac{(a)_s}{\Gamma(b+s)}z^s\,.
\]
Here $(a)_s$ denotes the Pochhammer symbol, $(a)_0=1$ and 
\[
(a)_s=a(a+1)\cdots(a+s-1)\quad \text{for $s\geq 1\,$.}
\]
It follows that a solution $u$ to the eigenvalue equation 
$\mathcal P_m u=\lambda u\,$, that is regular at $r=0\,$, is given by
\[
u(r)=e^{-Br^2/4h}r^m \,\mathbf M\Bigl(-\frac{\lambda}{2hB},m+1,\frac{Br^2}{2h}\Bigr)\,.
\]
Since we are interested in the Dirichlet case, the eigenvalues are determined by 
the condition $u(R)=0\,$, giving us the equation
\begin{equation}
\label{eq:eigeqm}
\mathbf M\Bigl(-\frac{\lambda}{2hB},m+1,\frac{BR^2}{2h}\Bigr)=0\,.
\end{equation}
Thus, given $b\in\mathbb{N}\setminus\{0\}$ we are interested in zeros of 
$a\mapsto \mathbf M(a,b,z)$ for large positive~$z$.
\begin{lemma}[{\cite[§13.9]{nist}}]
\label{lem:Mzeros}
Let $p(a,b)$ denote the number of positive zeros to the function 
$ \mathbb R \ni z \mapsto \mathbf M(a,b,z)\,$. Then
\begin{enumerate}[i),leftmargin=*]
\item $p(a,b) =0$ if $a\geq 0$ and $b \geq 0\,$.
\item $p(a,b) =\lceil - a \rceil$ if $a <0$ and $b \geq 0\,$. Here
$\lceil x\rceil$ denotes the smallest integer greater than or equal to $x$.
\end{enumerate}
\end{lemma}
It follows that the equation~\eqref{eq:eigeqm} has no solutions for 
$\lambda<0$ (which we already know, since $P_-$ is non-negative) and 
at most $k$ solutions in the interval $0<\lambda<2khB$. We conclude
that $\lambda_{m,k}^D(h)\geq 2khB$.

\subsubsection*{Proof of the upper bound in~\ref{it:c}}
We work here with $\mathcal P_m$\,, $m\geq 0\,$. We will use as trial state
\[
v(r)=r^m\bigl(e^{BR^2/4h-Br^2/4h}-e^{Br^2/4h-BR^2/4h}\bigr).
\]
A calculation shows that, as $h\to 0^+\,$,
\[
\|v\|^2=\int_0^R v(r)^2\,r\,dr=2^m m!\Bigl(\frac{h}{B}\Bigr)^{m+1}e^{BR^2/2h}+\mathcal O(1)
\]
and
\begin{equation}
\label{eq:Pmv}
\mathcal P_mv=2(m+1)hBr^me^{Br^2/4h-BR^2/4h}\,.
\end{equation}
Multiplying $\mathcal P_m v$ with $v$ and integrating we find that, 
as $h\to0^+$,
\[
\langle v,\mathcal P_m(h)v\rangle
=hBR^{2(m+1)}+\mathcal O(h^2)\,.
\]
This gives the lower bound, as $h\to 0^+\,$,
\[
\lambda_{m,0}^D(h)\leq \frac{\langle v,\mathcal P_m(h)v\rangle}{\|v\|^2}
=\frac{2hB(BR^2/2h)^{m+1}}{m!}e^{-BR^2/2h}(1+\mathcal O(h))\,.
\]

\subsubsection*{Proof of the lower bound in~\ref{it:c}}
We will use the Temple inequality, saying that
\[
\lambda_{m,0}^D(h)\geq \eta-\frac{\epsilon^2}{\beta-\eta}\,,
\]
with
\[
\eta=\frac{\langle v,\mathcal P_m(h)v\rangle}{\|v\|^2},\quad
\epsilon^2=\frac{\|\mathcal P_m v\|^2}{\|v\|^2}-\eta^2,\quad
\beta=2hB\leq\lambda_{m,1}^D(h)\,.
\]
From~\eqref{eq:Pmv} we get
\[
\|\mathcal P_m v\|^2=4(1+m)^2h^3BR^{2m}(1+\mathcal O(h))\,.
\]
From the lower bound of the second eigenvalue, we take $\beta=2hB$.
Hence, if $h$ is sufficiently small, 
\[
\frac{1}{\beta-\eta}=\frac{1}{2hB}(1+\mathcal O(h^{+\infty})).
\]
In $\epsilon^2$ the term $\eta^2$ is negligible in comparison with 
the first term. Thus, as $h\to0^+\,$,
\[
\frac{\epsilon^2}{\beta-\eta}
=\frac{2h}{BR^2}(1+m)^2\frac{2hB(BR^2/2h)^{m+1}}{m!}e^{-BR^2/2h}(1+\mathcal O(h))\,.
\]
Thus, compared with $\eta$ this term has an extra power of $h$, and 
thus can be considered as an error term, and we get the lower bound
as $h\to 0^+$
\begin{equation}
\lambda_{m,0}^D(h)\geq 
\frac{2hB(BR^2/2h)^{m+1}}{m!}e^{-BR^2/2h}(1+ \mathcal O(h))\,.
\end{equation}
This completes the proof of Theorem~\ref{thm:Pauli-m}.
\end{proof}

\begin{remark}
In fact, the calculations in the proof of Theorem~\ref{thm:Pauli-m} 
suggest that if $\lambda^D_{P_-,0}(h)\leq 
\lambda^D_{P_-,1}(h)\leq\cdots$ denotes the increasing sequence of eigenvalues
of $P_-$ then
\[
\lambda^D_{P_-,m}(h)=2hB \frac{(BR^2/2h)^{m+1}}{m!} e^{-BR^2/2h}(1+\mathcal O(h))\,.
\]
This would have given an alternative proof of Erd\H{o}s's result \cite{E1} but 
the uniform control in $m$ is missing in the estimates given by Theorem~\ref{thm:Pauli-m}.
\end{remark}

\section{A Faber-Krahn type  Erd\H{o}s inequality }\label{ss6.1}
We recall one by L.~Erd\H{o}s (\cite{E1,E2})  which is useful in this context.
\begin{proposition}\label{prop:erd}
For any planar domain $\Omega$ and $B  >0$, let $\lambda^{D}(h, B , \Omega)$ be 
the ground state energy of the Dirichlet realization of the semi-classical 
magnetic Laplacian with constant magnetic field equal to $B$ in $\Omega$. Then 
we have:
\begin{equation}\label{ineq}
\lambda^{D}(h,B , \Omega) \geq \lambda^{D} (h,B , D(0,R))\;,
\end{equation}
where $D(0,R)$ is the disk with same area as $\Omega\,$:
\[
\pi R^2 = \Area(\Omega)\,.
\]
Moreover the equality in (\ref{ineq})  occurs if and  only if $\Omega = D (0,R)\,$.
\end{proposition}

Hence we can combine this proposition with the optimal lower bound obtained 
for the disk (with $m=0$).

\section{Properties of the scalar potential---the torsion problem}
\label{sec:torsion}

In this section, we revisit a discussion of \cite{EKP}  about relating the decay 
rate with either the geometry of $\Omega$ or the properties of the magnetic 
field. We would like to see if the fact that we have the optimal $\psi_{min}$ 
can lead to the improvements of the statements of \cite{EKP} by using classical 
results in the theory of the torsion function. In this section we assume that 
$B(x)=1$, so $-2\psi  =f_\tau $ is the torsion function in $\Omega$, for which we 
collect available upper bounds. Our main reference is the book by 
R. Sperb~\cite{Sp}.

\subsection{Saint--Venant torsion problem and application}\label{s3.1}

We refer here to Example~3.4 in \cite{K2}. Let $\Omega \subset \mathbb R^n$ be 
a strictly convex domain and let $f_\tau$ be the solution of the Saint--Venant torsion problem
\begin{equation}
\begin{cases}
\Delta f_\tau  = -2 &\text{ in } \Omega\\
f_\tau= 0 &\text{ on } \partial \Omega.
\end{cases}
\end{equation}
The function $f_\tau$ is also known as warping function. For $n=2$ it is known from \cite{K1}  that the square root of $f_\tau$ is a strictly concave function. For $n\geq 2$  the concavity follows from Lemma 3.12 and Theorem 3.13 in 
\cite{K2}. As a consequence, we get
\begin{proposition}\label{propkaw}
If $\Omega$ is strictly convex and $B(x)=1$, then there exists a unique $x_0$ such that $\psi(x_0) =\psi_{min}$\,.
\end{proposition}
Note that at $x_0$, the Hessian of $\psi$ is definite positive. This can be deduced from the information on $f_\tau =-2 \psi$ given by Kawohl's result. At $x_0$, we have indeed
\[
\Hess \psi (x_0) = -  f_\tau^\frac 12  (x_0)\,  \Hess f_\tau ^\frac 12 (x_0)\,,
\]
and the trace of the Hessian of $\psi$ at $x_0$ is positive.

If $\Omega$ is not  convex, we can loose the uniqueness of $x_0$ as for example in the case of the dumbell (see Figure \ref{fig:dumbell}).

\begin{figure}[ht]
\centering
\includegraphics[width=5cm]{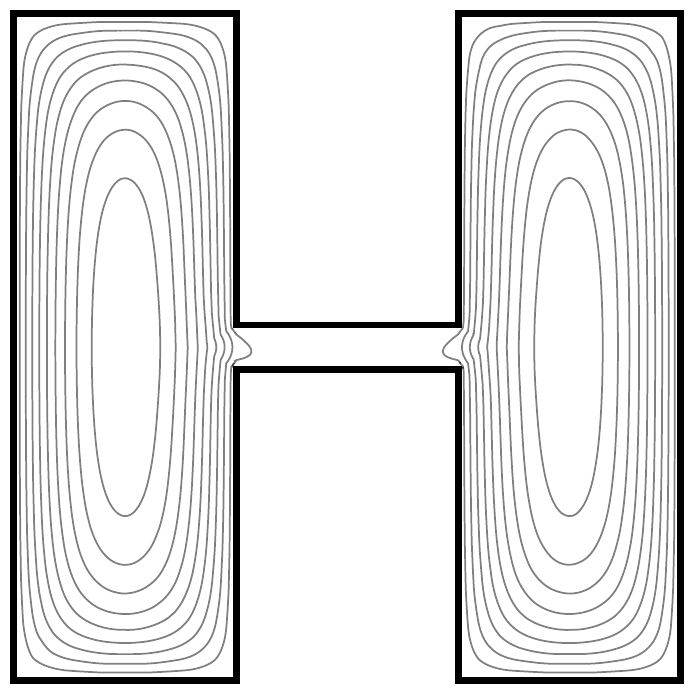}
\caption{Levels of $\psi$ in the case of the dumbell.}
\label{fig:dumbell}
\end{figure}

\begin{remark}
\begin{itemize}[--,leftmargin=*]
\item[]
\item Some literature in Mechanics (see for example \cite{FNJ}) claims  the concavity 
of the torsion function $f_\tau $ in the strictly convex case (at least when $f_\tau $ is 
a polynomial) but this is obviously wrong in the case of the equilateral 
triangle (see next subsection). 
\item It is natural to ask the same question 
for solutions of $\Delta u =- 2 B(x)$   for $B>0$ non-constant.  We have no 
condition on $B$ to propose  implying the uniqueness  
for the maximum of $u$. We have
only verified that the argument given in~\cite{K2} in the case of constant
$B$ breaks down as soon as $B$ is not constant.
\end{itemize}
\end{remark}

\subsection{Examples}\label{s3.2}

Below we discuss the situation in some simple domains. We refer to~\cite{FNJ} 
for many other examples of domains for which explicit (or semi-explicit) 
computations can be done.  In Figure~\ref{fig:contours}, we show the level sets 
obtained for $\psi$ using Mathematica.

\begin{figure}[ht]
\centering
\includegraphics[width=4cm]{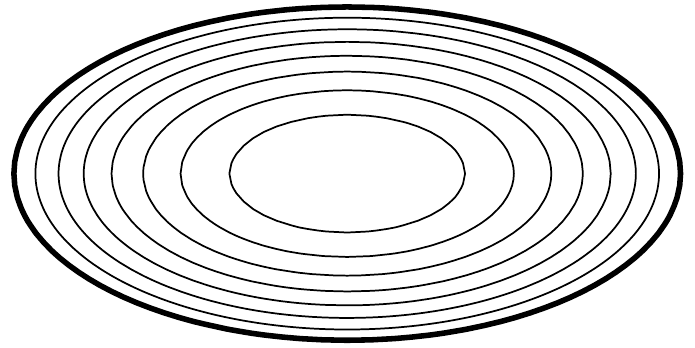}
\hskip1cm
\includegraphics[width=4cm]{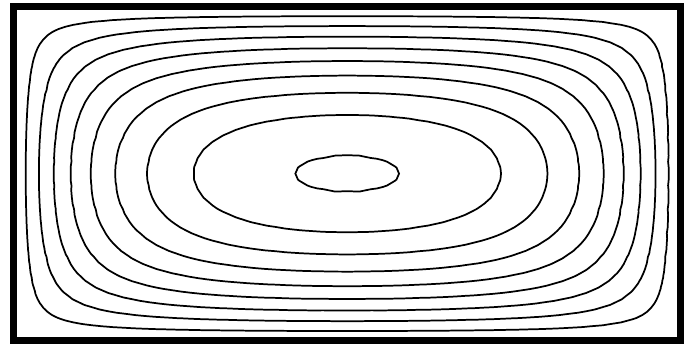}\\[3ex]
\includegraphics[width=4cm]{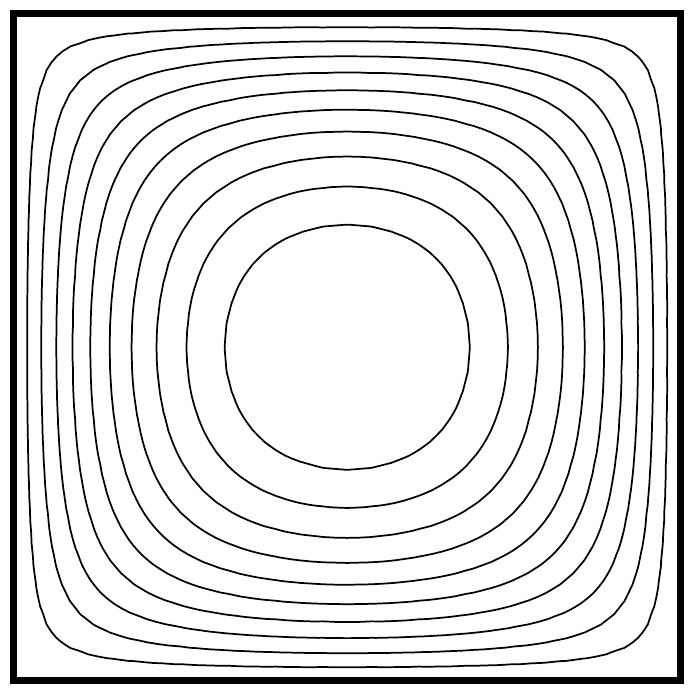}
\hskip1cm
\includegraphics[width=4cm]{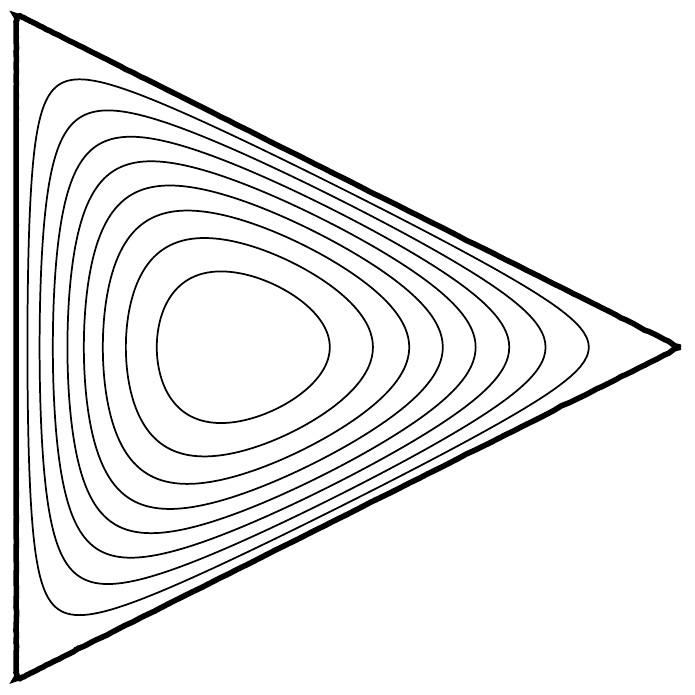}
\caption{Levels of $\psi$ in the case of the ellipse, rectangle, 
square and triangle.}
\label{fig:contours}
\end{figure}

\subsubsection{The disk}
 If $\Omega = D(0,1)$ and $B=1$, we have 
 \begin{equation}\label{psiminopt}
 \psi(x) = \frac 14 | x|^2  - \frac 14\,, \quad \psi_{min} = - \frac 14\,.
 \end{equation} 
Theorem~\ref{th2.1} gives  a rather accurate lower bound but see 
Section~\ref{sec:disk}  for the optimal result.
 
\subsubsection{The ellipse}  If $\Omega=\bigl\{(x,y)\in \mathbb R^2\,,\, \frac{x^2}{a^2} + \frac{y^2}{b^2} < 1\bigr\}$. then, for $B=1\,$,
 \begin{equation}\label{ellipse}
 \psi (x,y) = \frac{ 1}{ 2/a^2 + 2/b^2} \, \Bigl( \frac{x^2}{a^2} + \frac{y^2}{b^2} -1 \Bigr)\,,
 \quad
 \psi_{min} = - \frac{ 1}{ 2/a^2 + 2/b^2} \,.
 \end{equation}

\subsubsection{The rectangle}
For the rectangle  $(-\frac a2,\frac a2)\times( -\frac b2, \frac b2)$, 
we have an explicit solution using Fourier series (see~\cite{FNJ} or
this course\footnote{\url{http://people.inf.ethz.ch/tulink/FEM14/Ch1_ElmanSyvesterWathen_Ox05.pdf}}):
\begin{equation} \label{rect}
\psi_{a,b} (x,y) = - \sum_{(k,\ell)\in \mathbb N_{odd}^2} \frac{(-1)^{\frac{k+\ell}{2} -1}2^4 a^2b^2}{\pi^4k\ell (k^2 b^2+\ell^2 a^2)} \, \cos \frac {k\pi x}{a}\, \cos \frac{\ell \pi y}{b}.
\end{equation}
In the case of the square (with $a=b=2$) one has $\psi_{2,2}(0,0)\approx-0.294685$. 
In the limit $b \to +\infty$ one recovers the argument of~\cite{EKP}, who use
the function $\psi$ corresponding to the band of size $a$:
\[
\psi_{a,\infty} (x) =\frac{1}{2} \Bigl(x^2- \Bigl(\frac a2\Bigr)^2\Bigr)\,,\quad \psi_{a,\infty,min} =- \frac{a^2}{8}\,.
\]

\subsubsection{The equilateral triangle}
For the equilateral triangle, there is an explicit formula given by
\begin{equation}\label{equil}
 \psi (x,y)  = -\frac{1}{4a}\,  \Bigl(x-\sqrt{3} y - \frac 23 a\Bigr) 
\Bigl(x+\sqrt{3} y - \frac 23 a\Bigr) 
\Bigl(x+\frac 13 a\Bigr)\,,
\quad \psi_{min} = - \frac{a^2}{27}\,,
\end{equation}
the minimum of $\psi$ is attained at $(0,0)$\,.

%\section{Further estimates for $\psi_{min}$ in the case $B=1$.}

 \subsection{Known bounds for the torsion function}
The problem of the torsion of an elastic beam is explained in \cite[p.~3]{Sp}. 
But the main 
properties are analyzed in Chapter~6 (mainly Section~6.1). For a given 
$\Omega$, we can define the diameter $\delta (\Omega)$, the maximal 
width\footnote{$\ell (\Omega)$ is the maximum (over the directions) of the 
maximal width in one direction} $\ell (\Omega)$ and the inner radius 
$\rho (\Omega)$. We have the following estimates:
\begin{align}
\label{delta}
\psi_{min} &\geq -\frac{\delta (\Omega)^2}{4}\,,\\
\label{ell}
\psi_{min} &\geq -\frac{\ell(\Omega)^2}{8}\,,\\
\shortintertext{and, in the convex case,}
\label{rho}
\psi_{min} &\geq -\frac{\rho(\Omega)^2}{2}\,.
\end{align}
Inequality~\eqref{ell} permits to recover the result of~\cite{EKP} 
and~\eqref{rho} in the convex case is better than the corresponding statement 
in~\cite{EKP}. There are more accurate estimates, when $\Omega$ has two axes of 
symmetry, see equation~(6.14) in~\cite{Sp}.

It is also interesting to mention the inequality due, according to the book of 
R.~Sperb~\cite[p.~193]{Sp},  to P\'olya--Szeg\"o~\cite{PS} (1951):
\begin{equation}\label{ineqPS}
 \psi_{min} (\Omega) \geq \psi_{min} (D(0,R))\,,
\end{equation}
where $R$ is the radius of the disk of same area as $\Omega$. This result is 
coherent with the Erd\H{o}s result recalled in Proposition \ref{prop:erd}.

\subsection{General comparison statements  for $\psi_{min}$}
In order to have lower bounds or upper bounds for $\psi_{min}$, one way, using 
the maximum principle, is to either find $ \psi^{sub}$ such that
\begin{equation} \label{sub}
\Delta \psi^{sub} \geq B\text{ in } \Omega\,;\,    \psi^{sub} \leq 0\text{ on } \partial \Omega\,,
\end{equation}
which will imply
\begin{equation}\label{lwsub}
\psi^{sub}_{min} \leq \psi_{min}\,,
\end{equation}
or  to find $\psi^{sup}$ 
\begin{equation}\label{sup}
\Delta  \psi^{sup} \leq B\text{ in } \Omega\,;\,  \psi^{sup} \geq 0\text{ on } \partial \Omega\,,
\end{equation}
which will imply
\begin{equation}\label{lwsup}
\psi^{sup}_{min} \geq \psi_{min}.
\end{equation}

This  can be used in 
different ways:
\begin{itemize}[--,leftmargin=*]
\item In the case, when $B(x) >0$, use the comparison  between the magnetic 
field $B(x)$ and the constant magnetic field $\inf B$ and $\sup B$.  
\item Recover Theorem 2.2  in~\cite{EKP} but \eqref{ell} is more direct. 
\end{itemize}
Note that this gives for $\psi_{min}$ a monotonicity result with respect to 
$B$ which is not necessarily true for the Pauli operator itself.

\subsection{The results by C.~Bandle}
In~\cite{Ba} C.~Bandle uses isoperimetric techniques and the conformal mapping
theorem. The domain $\Omega$ is supposed to be simply connected
and $B>0$ is supposed to satisfy
\[
\Delta\log B(x)+2C B(x)\geq 0\quad\text{in }\overline{\Omega}
\]
for some constant $C\in\mathbb R\,$. If $u$ solves
\[
\begin{cases}
\Delta u = -B(x)&\text{in $\Omega$}\\
u = 0 & \text{on $\partial\Omega$}\,,
\end{cases}
\]
and if $CM<4\pi$, where $M=\int_\Omega B(x)\,dx$ is the flux of the magnetic
field, then $u$ satisfies the bound
\[
u(x)\leq \frac{1}{C}\log\frac{4\pi}{4\pi-CM}\,.
\]
with equality when $\Omega$ is the disk, and $B(x)=(1+C|x|^2/4)^{-2}$ (or a
situation conformally equivalent to this). When $C=0$, the result reads 
\[
u(x)\leq \frac{M}{4\pi}\,,
\]
and we recover~\eqref{ineqPS} in the case $B>0$ constant.

In~\cite{ScSp,Sp1} one can find other estimates, involving curvature terms.

\subsection{Torsion, lifetime  and Hardy inequality}
In the twodimensional case, one should also mention (see~\cite[Theorem~1]{BC1}) 
the following estimates:
\begin{proposition} 
Let us assume that $\Omega$ is simply connected in $\mathbb R^2$. 
With $\psi$ solution of  $\Delta \psi =1$ in 
$\Omega$, $\psi_{/\partial \Omega}=0$, we have
\begin{equation}\label{eq:vdbc}
\lambda(\Omega)^{-1} \leq  |\psi_{min}| 
\leq  \frac{7\zeta(3)}{16} {\bf j}^2  \lambda(\Omega)^{-1}\,.
\end{equation}
\end{proposition}
Note that
\[
\frac{7\zeta(3)}{16} {\bf j^2} < 0.5259   
\times   {\bf j}^2 < 0.5259\times 5.784025 < 3.0419\,.
\]
One can then combine with Hardy's inequality 
(see~\cite{An, BC1, BC2,Dav,LS,vdBC, Her}) but this does not seem to give 
significative improvements in our twodimensional situation when we compare with 
what is given by~\eqref{ell} or~\eqref{rho}. $|\psi_{min}|$ is also related to 
the maximal expected lifetime of a Brownian motion (see \cite{BC1,BC2}).

\bibliographystyle{plain}

\end{document}